\documentclass{article}

\usepackage[french]{babel}
\usepackage[utf8]{inputenc}
\usepackage[T1]{fontenc}
\usepackage{hyperref}

\usepackage{amsmath}
\usepackage{amsthm}
\usepackage{amsfonts}
\makeatletter
\def\amsbb{\use@mathgroup \M@U \symAMSb}
\makeatother
\usepackage{bbold}
\DeclareSymbolFont{bbold}{U}{bbold}{m}{n}
\DeclareSymbolFontAlphabet{\mathbbold}{bbold}

\def \AA {\amsbb A}
\def \CC {\amsbb C}
\def \HH {\amsbb H}

\def \NN {\amsbb N}
\def \QQ {\amsbb Q}
\def \RR {\amsbb R}

\def \ZZ {\amsbb Z}
\def \ind {\mathbb 1}
\def \ii {\iota}
\def \GL {\mathrm{GL}}

\theoremstyle{remark}
\newtheorem{remark}{Remarque}
\theoremstyle{plain}
\newtheorem{lemma}{Lemme}
\newtheorem{proposition}{Proposition}

\title{Distributions de séries d'Eisenstein presque holomorphes sur un
  corps totalement réel}

\author{Julien Puydt\\
Institut Joseph Fourier UMR5582\\
Grenoble, France\\
E-mail:\texttt{julien.puydt@ujf-grenoble.fr}
}

\begin{document}

\maketitle

\begin{abstract}
  On définit d'abord un type de séries d'Eisenstein, sous forme de
  distributions fournissant des formes automorphes presque-holomorphes
  sur un corps de nombre totalement réel, en donnant différentes
  expressions (intégrale, sommatoire) ; on montre ensuite qu'elles
  vérifient bien les propriétés attendues, puis on calcule les
  coefficients de Fourier explicitement. On termine par un calcul
  explicite dans le cas le plus simple, qui fait apparaître que les
  objets abstraits considérés ici généralisent bien les objets
  classiques.
\end{abstract}

\section{Introduction}\label{sec:intro}

\subsection{Résumé}

Ma thèse de doctorat~\cite{Puydt2003a} portait sur la déter\-mination
de congru\-ences sur les valeurs spéciales des fonctions $L$ de
certaines formes modulaires, connaissant des congruences sur leurs
coefficients de Fourier. Les résultats étaient établis dans un cadre
semi-adélique sur le corps des nombres rationnels.

Le principe de base est d'utiliser la méthode de Rankin-Selberg pour
obtenir des congruences sur les valeurs spéciales à partir de
congruences sur les coefficients de Fourier. Cette méthode nécessite
de construire des séries d'Eisenstein adaptées au cadre dans lequel on
souhaite travailler. Le but du présent travail est donc la
construction d'une variante des séries d'Eisenstein pour généraliser
le travail de ma thèse au cas de corps de nombres totalement
réels.

La construction présentée ici possède la combinaison suivante de
caractéristiques~:
\begin{itemize}
\item elle se présente sous forme de distributions, car la liberté de
  choisir de bonnes fonctions-tests donne plus de souplesse pour les
  applications ;
\item le cadre est adélique (ou semi-adélique), car on peut alors
  faire intervenir des caractères de conducteur relativement
  arbitraire ;
\item on considère des formes automorphes, pour se rapprocher de la
  théorie des représentations ;
\item une régularité moins restrictive que l'holomorphie, mais sans
  trop s'en éloigner pour garder une structure rationnelle~: la
  presque-holomorphie.
\end{itemize}

On termine le travail par un exemple extrêmement concret qui montre
que les séries les plus classiques sont un sous-cas, obtenu en faisant
des choix de paramètres et de fonction-test simples.

\subsection{Rapide historique}

La littérature regorge d'exemples de séries d'Eisenstein, possédant
diverses combinaisons de bonnes propriétés (en plus de celles qui \og
caractérisent\fg{} ce type d'objets) ; passons quelques exemples en
revue.

Shimura, dans son article~\cite{Shimura1983a}, travaille sur un corps
de nombres totalement réel (et par moments une extension totalement
imaginaire de ce dernier), mais avec des formes modulaires classiques
(pas dans un cadre de représentations), sur des groupes symplectiques
et unitaires. Enfin, les séries d'Eisenstein considérées ne prennent
pas la forme de distributions.

Hida, dans son article~\cite{Hida1985a}, utilise une série
d'Eisenstein comme une mesure (bornée), mais le cadre est limité à
$\QQ$, et il applique des opérateurs différentiels de Shimura pour
obtenir des formes modulaires presque-holomorphes (bien qu'il ne les
appelle pas ainsi), auxquelles il applique ensuite un opérateur de
projection holomorphe.

Dans le livre~\cite{Panchishkin1991a} de Panchishkin, en section~4.1,
on trouve présentée une série d'Eisenstein dans le cadre des formes
modulaires de Hilbert, dont le développement en série de Fourier est
explicité en proposition~4.2. Le travail est bien sur un corps
totalement réel et des distributions sont construites comme définies
sur un groupe de Galois. L'approche n'est pas celle de la théorie des
représentations.

Dans son article~\cite{Scholl1998a}, Scholl, en section~4.3, page~433,
présente une construction sous forme de distributions dans un cadre
de formes automorphes adéliques, mais limitée aux adèles sur $\QQ$, et
à des formes modulaires holomorphes.

Finalement, dans ma thèse~\cite{Puydt2003a}, les séries d'Eisenstein
étaient des distributions adéliques et fournissait des formes
automorphes presque-holo\-mor\-phes, mais le corps de base était $\QQ$.

\subsection{Résultats}

Pour une fonction de Schwartz-Bruhat $\varphi$, on définit
(en~\ref{sec:def_int}) la série d'Eisenstein associée par la formule~:
\[
E^{\rm ana}_{\omega,s}(\varphi):g
\longmapsto
|\det g|^{-s}
\int_{\AA^\times/F^\times}\omega(t)|t|^{2s}\Theta(\varphi)(g/t){\rm d}^\times t
\]
où~:
\[
\Theta(\varphi):g\longmapsto\sum_{v\in F^2\backslash\{0\}}(g\varphi)(v)
\]
dont on montre la convergence sous réserve que la partie réelle de $s$
soit assez élevée en~\ref{sec:conv}.

On montre en~\ref{sec:analytique} que cette définition se prolonge de
façon méromorphe en $s$, avec d'éventuels pôles simples en $1$ et $0$,
avec une ébauche d'équation fonctionnelle ; qui est ensuite affinée en
l'équation fonctionnelle suivante en~\ref{sec:eqfun}~:
\[
E^{\rm ana}_{\omega,s}(\varphi)(g)
=
E^{\rm ana}_{\bar\omega,1-s}(\hat\varphi)(g^\sharp)
\]
où $\hat\varphi$ est la transformée de Fourier (symplectique) de
$\varphi$ et $g^\sharp=\frac{1}{\det g}g$ (deux involutions).

Ces bonnes propriétés analytiques établies, on prouve alors que la
fonction $g\mapsto E_{\omega,s}^{\rm ana}(\varphi)(g)$ a de bonnes
propriétés d'automorphie~:
\begin{itemize}
\item elle est $\GL_2(F)$-invariante à gauche (\ref{sec:autom_F_inv}) ;
\item $\omega$ est un caractère central (\ref{sec:autom_centre}) ;
\item de bons choix pour les parties archimédiennes donnent une notion
  de poids au sens des formes automorphes (\ref{sec:autom_poids})~:
  \[
  E_{\omega,s}^{\rm ana}(\varphi)(g\rho)
  =
  \left(\prod_{\ii\in I}e^{ik_\ii\theta_\ii}\right)
  E_{\omega,s}^{\rm ana}(\varphi)(g)
  \]
  où~:
  \[
  \rho=\left(
    \begin{pmatrix}
      \cos(\theta_\ii) & \sin(\theta_\ii) \\
      -\sin(\theta_\ii) & \cos(\theta_\ii)
    \end{pmatrix}
  \right)_{\ii\in I}
  \]
\item il existe un sous-groupe de $\GL_2(\hat{\mathcal O})$ d'indice
  fini pour lequel elle est invariante à droite
  (\ref{sec:autom_niveau}) ;
\item la croissance est modérée (\ref{sec:autom_cr_moderee}).
\end{itemize}

On calcule très explicitement la fonction de Whittaker de la forme
automorphe, et on prouve la formule suivante (en
page~\pageref{frml:Whit_int})~:
\[
a_1(g)=
|\det g|^{-s}
\int_{\AA\times\AA^\times}\omega(t)|t|^{2s}
(g\varphi)\left(t\binom{u}1\right) \Psi(u) {\rm d}u{\rm d^\times}t
\]

Cette formule ouvre la porte à deux types de considération~:
\begin{itemize}
\item d'une part, elle montre que si $\varphi$ se factorise suivant
  les différentes places, alors la fonction de Whittaker se factorise
  de même ;
\item d'autre part, elle permet d'étudier plus finement la régularité en
les places infinies ; on montre alors en ~\ref{sec:phol_cas} que pour
le choix $s=\frac{k}2-r$ avec $r\in\NN$, on obtient une forme
automorphe presque-holomorphe au sens de~\ref{sec:prelim_phol}.
\end{itemize}

\subsection{Applications}

La motivation principale à cette nouvelle construction est l'étude de
congru\-ences pour les valeurs spéciales des fonctions $L$ des formes
modulaires, pour généraliser les résultats obtenus dans ma
thèse~\cite{Puydt2003a} sur un corps totalement réel.

À partir de congruences sur les coefficients de Fourier, la méthode de
Rankin-Selberg permet d'obtenir des congruences sur les valeurs
spéciales. Cette notion de congruence nécessite en préliminaire que
les objets considérés vérifient des propriétés d'algébricité, et que les
espaces admettent des structures rationnelles.

De tels résultats sont déjà connus dans des cadres proches. Par
exemple, dans le cas des formes modulaires de Hilbert sur un corps
totalement réel, Shimura~\cite{Shimura1978a} a montré l'algébricité
des valeurs spéciales. On peut aussi citer le
livre~\cite{Panchishkin1991a} de Panchishkin, qui montre des
congruences entre valeurs spéciales, mais utilise la méthode de la
projection holomorphe, qui introduit des complications.

On sait que l'on peut contourner ces difficultés via la méthode de la
projection canonique, qui a déjà été utilisée pour obtenir de bonnes
congruences dans divers cadres (\cite{Panchishkin2002b} et
\cite{Puydt2003a}). Il reste à l'étendre au cas des formes automorphes
sur un corps totalement réel qui nous intéresse.

Dans un second temps, on peut penser que l'obtention de telles
congruences pour des fonctions $L$ complexes permettra d'obtenir des
renseignements sur des fonctions $L$ $p$-adiques ; par exemple en
utilisant des méthodes similaires au travail de
Manin~\cite{Manin1976a}, ou celui de Dabrowski~\cite{Dabrowski1994a}.

Enfin, comme le calcul des coefficients de Fourier des séries
d'Eisenstein développées ici est assez explicite, il pourra être
intéressant d'en étudier plus finement les dénominateurs pour les
contrôler, en s'approchant de résultats d'intégralité, par exemple
dans le cas de séries d'Eisenstein-Klingen (l'algébricité a déjà été
étudiée par Michael Harris durant les années 1980).

\subsection{Plan du travail}

Après cette introduction, on trouvera quelques
préliminaires~\ref{sec:prelim}, pour fixer les notations d'abord, pour
rappeler rapidement quelques résultats classiques d'analyse harmonique
sur les adèles, puis pour présenter une fonction confluente
hypergéométrique et deux calculs d'intégrales. Ces préliminaires se
terminent par une assez longue discussion de la notion de
presque-holomorphie~\ref{sec:phol}, qui permet de définir une
structure rationnelle sur les formes automorphes
en~\ref{sec:phol_rationn}.

On définit alors les distributions proprement dites, d'abord via une
expression intégrale en~\ref{sec:def_int}, dont on montre d'abord la
convergence en~\ref{sec:conv}, avant d'en donner une expression
sommatoire en~\ref{sec:def_somme}.

On étudie ensuite le prolongement analytique et l'équation
fonctionnelle en~\ref{sec:ana_eq_fun}, avant de prouver une par une
les différentes propriétés d'automorphie en~\ref{sec:autom}.

Après, en section~\ref{sec:coeffs}, on discute des coefficients de
Fourier. On obtient une expression explicite de la fonction de
Whittaker en~\ref{sec:coeffs_Whittaker}, qui permet de conclure par la
section~\ref{sec:phol}, dans laquelle on établit la
presque-holomorphie des séries d'Eisenstein considérées sous certaines
hypothèses en~\ref{sec:phol_cas} ; avec des coefficients dont
l'algébricité est assez bien contrôlée.

Finalement, en section~\ref{sec:explicite}, on montre que si l'on fait
certains choix simples, la construction permet de retrouver les séries
d'Eisenstein classiques.

\section{Préliminaires}\label{sec:prelim}

\subsection{Corps de base et données afférentes}

$F$ est un corps de nombres totalement réel, et $I$ l'ensemble de ses
plongements réels.

On note $\mathcal O$ l'anneau des entiers de $F$, $\AA$ l'anneau des
adèles de $F$ ; on note $\AA_\infty$ la partie archimédienne des
adèles, identifiée à $\RR^I$, et $\AA_f$ la partie non-archimédienne ;
de façon plus générale, si $x$ est un objet adélique, on notera $x_f$
sa partie non-archimédienne et $x_\infty$ sa partie archimédienne, et
pour chaque place $v$ de $F$, on notera $x_v$ sa composante le long de
$v$ ; la complétion étant notée $F_v$. On écrira en particulier
$x_\ii$ pour désigner la partie archimédienne de l'objet $x$ pour le
plongement $\ii\in I$. Dans $\AA_f$, on notera $\hat{\mathcal O}$
l'anneau des entiers complétés ; on sait que $\AA_f=\hat{\mathcal
  O}\otimes F$.

On notera $\AA_1^\times$ les idèles de norme $1$, et de façon
similaire $\AA_{<1}^\times$ (respectivement $\AA_{>1}^\times$) les
idèles de normes strictement plus petite (respectivement strictement
plus grande) que $1$.

\subsection{Notations matricielles}

Pour tout anneau $R$, $\mathcal{M}_2(R)$ est l'algèbre des matrices
$2\times2$ à coefficients dans $R$, et $\GL_2(R)$ est son sous-groupe
des matrices inversibles ; de manière générale, si $m$ est une telle
matrice, on notera~:
\[
m=\begin{pmatrix}a_m & b_m \\ c_m & d_m\end{pmatrix}
\]

Dans ce sous-groupe, on considèrera~:
\begin{itemize}
\item $B(R)$ les matrices boréliennes, telles que $c_m=0$ ;
\item $U(R)$ les matrices de translation, qui sont boréliennes avec
  $a_m=c_m=1$ ;
\item $T(R)$ les matrices diagonales, telles que $b_m=c_m=0$ ;
\item $Z(R)\subset T(R)$ le centre, telles que $a_m=d_m$ ;
\item si $\mathfrak n$ est un idéal de $R$, on notera $B(\mathfrak n)$
  les matrices $m$ inversibles telles que $c_m\equiv0\bmod\mathfrak n$
  ;
\item dans les mêmes conditions, $\Gamma(\mathfrak n)=B(\mathfrak
  n)\cap SL_2(R)$.
\end{itemize}

Ces dernières notations seront en particulier utile dans
$\hat{\mathcal O}$ -- l'anneau sera alors implicite.

\subsection{Analyse harmonique sur les adèles}

La référence usuelle pour cette section est bien évidemment la thèse
de Tate~\cite{Tate1965b} ; elle est aussi présentée (avec une
bibliographie assez riche) dans l'article de survol de
Kudla~\cite{Kudla2003a}. Le livre de Ramakrishnan et
Valenza~\cite{RamakrishnanValenza1999a} est une bonne introduction
(plus long). Bien sûr, Weil~\cite{Weil1967a} couvre tout en détail.

L'espace vectoriel des fonctions de Schwartz-Bruhat définies sur
$\AA^2$ et à valeurs dans $\CC$ sera noté $\mathcal S(\AA^2)$.

Toutes ne s'écrivent pas sous la forme
$x\mapsto\varphi_f(x_f)\varphi_\infty(x_\infty)$, où $\varphi_f$ est
localement constante à support compact et $\varphi_\infty$ de classe
$\mathcal C^\infty$ à décroissance rapide ainsi que toutes ses
dérivées, mais les fonctions de ce type engendrent vectoriellement
l'espace ; on pourra donc souvent s'y ramener.

On fixe le caractère additif $\Psi:\AA\rightarrow\CC$ usuel et on
choisit les normalisations usuelles des mesures de Haar pour que les
formules d'inversion de la transformée de Fourier fonctionnent. En
particulier, ce caractère additif est défini par produit sur les
places ; avec sur les places finies $v|p$ une expression
$x\mapsto\exp(-2i\pi{\rm Tr}_{F_v/\QQ_p}(x))$, et en les places
infinies une expression $x\mapsto\exp(2i\pi x)$ (on rappelle que
toutes les places infinies sont réelles).

On définit la transformée de Fourier (symplectique) d'une fonction de
Schwartz-Bruhat $\varphi$ par la formule:
\[
\hat\varphi:(x_1,x_2)
\mapsto
\int_{\AA^2}\varphi(y_1,y_2)\Psi(x_1y_2-x_2y_1){\rm d}y_1{\rm d}y_2
\]

Cette opération définit une bijection sur $\mathcal S(\AA^2)$, pour
laquelle on dispose du résultat suivant, conséquence de la formule de
Poisson~:
\[
\forall t\in\AA^\times,
\sum_{v\in F^2}\varphi(tv)=|t|^{-2}\sum_{v\in F^2}\hat\varphi(v/t)
\]

Le résultat suivant, classique, est énoncé et prouvé par exemple dans
la thèse de Tate~\cite{Tate1965b} (page 337), et dans le livre de
Ramakrishnan et Valenza~\cite{RamakrishnanValenza1999a} (page 283)~:
\begin{lemma}\label{lem:vol_ideles_1}
  \[
  {\rm Vol}(\AA_1^\times/F^\times)
  =
  \frac{2^{r_1}(2\pi)^{r_2}h_F R_F}{w_F\sqrt{|d_F|}}
  \left( =-{\rm Res}_1\zeta_F \right)
  \]
  où $r_1$ est le nombre de plongements réels du corps, $r_2$ est le
  nombre de plongements complexes à conjugaison près (dans cet
  article: $0$), $h_F$ est le nombre de classes, $R_F$ est le
  régulateur, $d_F$ le discriminant, $w_F$ le cardinal du groupe des
  racines de l'unité de $F$ et $\zeta_F$ la fonction $\zeta$ de
  Dedekind du corps $F$.
\end{lemma}

De là, on déduit le résultat suivant~:
\begin{lemma}\label{int_ideles_str_1}
  Si $\omega$ est un caractère de Hecke unitaire et $s\in\CC$ est de
  partie réelle strictement positive, alors~:
  \[
  \int_{\AA_{<1}^\times/F^\times}\omega(t)|t|^s{\rm d}^\times t
  =
  \frac{\delta(\omega)}{s}
  \]
  où~:
  \[
  \delta(\omega) = \left\{
    \begin{array}{ll}
      0
      & \text{si }\omega\text{ est ramifi\'e} \\
      {\rm Vol}(\AA_1^\times/F^\times)
      & \text{si }\omega\text{ n'est pas ramifi\'e}
    \end{array}
  \right.
  \]
\end{lemma}
\begin{proof}
  Il suffit essentiellement d'utiliser l'isomorphisme~:
  \[
  \AA_{<1}^\times/F^\times\simeq\AA_1^\times/F^\times\times]0;1[
  \]
  Comme on a supposé $\omega$ unitaire, il n'a pas de contribution sur
  la seconde composante ; comme par ailleurs $|.|$ est trivial sur la
  première, l'intégrale se factorise~:
  \[
  \int_{\AA_{<1}^\times/F^\times}\omega(t)|t|^s{\rm d}^\times t
  =\left( \int_{\AA_1^\times/F^\times}\omega(t){\rm d}^\times t \right)
  \left( \int_0^1t^s{\rm d}^\times t \right)
  \]
  d'où l'expression du résultat final.
\end{proof}

On fait agir à gauche une matrice $g$ de $\GL_2(\AA)$ sur une fonction
de Schwartz-Bruhat $\varphi$ via~:
\[
(g\varphi):x\longmapsto\varphi(g^{-1}x)
\]

On prendra soin de ne pas confondre cette action matricielle avec le
produit extérieur par un scalaire, pour lequel on dispose de la même
notation. En cas de doute, choisir l'interprétation qui donne un bon
résultat.

\subsection{Fonction confluente hypergéométrique}

La référence est ici l'article~\cite{Shimura1982a} de Shimura, qui
généralise des énoncés donnés dans~\cite{Shimura1975a}, qui crédite
Siegel pour le cas particulier qui nous intéresse.

L'article définit une fonction
$\omega:(z,\alpha,\beta)\mapsto\omega(z,\alpha,\beta)$ (en fait,
$\omega_m$, mais seul le cas $m=1$ nous intéresse ici), sur
$\HH'\times\CC^2$, où $\HH'$ est le demi-plan $\Re(.)>0$, dont
l'expression générale nous importe assez peu ; le théorème~3.1,
page~281, affirme que cette fonction est holomorphe sur son domaine de
définition, avec~:
\[
\forall(z,\alpha,\beta)\in\HH'\times\CC^2,
\omega(z,1-\beta,1-\alpha)=\omega(z,\alpha,\beta)
\]

La proposition~3.2, page~285 affirme que cette fonction a un
développement de type polynomial pour certaines valeurs des
paramètres, et l'article donne, en (3.17) et (3.18) (pour $n>0$, le
cas $n=0$ se calcule directement) un développement explicite qui
mérite d'être énoncé à part~:
\begin{proposition}\label{prop:conf_hyper_poly}
  Pour $(z,\alpha,\beta)\in\CC^3$ avec $\Re(z)>0$ et $n\in\NN$, on a~:
  \begin{eqnarray*}
    \omega(z,n+1,\beta)
    &=&
    \sum_{k=0}^n\binom{n}{k}\beta(\beta+1)\dots(\beta+k-1)z^{-k}
    \\
    \omega(z,\alpha,-n)
    &=&
    \sum_{k=0}^n\binom{n}{k}(1-\alpha)(2-\alpha)\dots(k-\alpha)z^{-k}
  \end{eqnarray*}
\end{proposition}

\subsection{Deux intégrales archimédiennes}

Le calcul de l'intégrale suivante, dont la preuve fait intervenir un
simple changement de variable, pour forcer l'apparition de la fonction
$\Gamma$ d'Euler, mérite d'être écrit~:

\begin{lemma}\label{lem:int_Gamma}
  Pour $a>0$ réel et $u$ complexe de partie réelle strictement
  supérieure à $-1$, on a~:
  \[
  \int_0^{+\infty} t^u\exp\left(-at^2\right){\rm d}^\times t
  =
  \frac12\frac{\Gamma\left(\frac{u}2\right)}{a^{\frac{u}2}}
  \]
\end{lemma}

Par ailleurs, le lemme suivant explique comment la fonction confluente
hypergéométrique entre en jeu dans ce travail~:
\begin{lemma}\label{lem:int_conf_hyper}
  Pour $y>0$ et $(\alpha, \beta)\in\CC^2$ tels que le premier membre
  converge, on a~:
  \begin{eqnarray*}
    \lefteqn{\int_\RR(x+iy)^{-\alpha}(x-iy)^{-\beta}\exp(-2i\pi x){\mathrm d}x}\\
    &=&
    \exp(-2\pi y)i^{\beta-\alpha}\Gamma(\alpha)^{-1}(2\pi)^{\alpha+\beta}
    (4\pi y)^{-\beta}\omega(4\pi y,\alpha,\beta)
  \end{eqnarray*}
\end{lemma}

\begin{remark}
  Plus loin, on va montrer le prolongement analytique des séries
  d'Eisenstein (et leur équation fonctionnelle) via la formule de
  Poisson ; le résultat précédent permettrait de montrer aussi ce
  résultat en prolongeant les coefficients de Fourier eux-mêmes. C'est
  la raison pour laquelle on n'insiste pas sur les conditions pour
  lesquelles l'intégrale converge~: il suffit de commencer le calcul
  dans un domaine où elle converge, et le second membre permet alors
  de généraliser le résultat et s'affranchir des limitations.
\end{remark}

\subsection{Presque-holomorphie}\label{sec:prelim_phol}

\subsubsection{Opérateurs différentiels}

On se contente ici de rappeler très rapidement ce dont on a besoin,
sans entrer dans les détails ; on en trouvera un peu plus dans
Bump~\cite{Bump1997a}.

Étant donnée une fonction $f$ de classe $\mathcal{C}^\infty$ sur
$\GL_2(\RR)$, on sait que l'on peut faire agir sur $f$ les éléments de
$\mathcal{M}_2(\RR)$ comme des opérateurs différentiels invariants à
gauche, via l'interprétation~:
\[
\mathcal{M}_2(\RR) = \mathfrak{gl}(2,\RR)
\]

Cette action se prolonge en une action de $\mathfrak{gl}(2,\CC)$, ce
qui permet de définir l'opérateur~:
\[
L = \frac12
\begin{pmatrix}
  1 & i \\
  i & -1
\end{pmatrix}
\]

Si on paramètre $\GL_2^+(\RR)$ par la décomposition d'Iwasawa~:
\[
\begin{array}{rcl}
  \RR^\times_+\times\RR\times\RR^\times_+\times\RR & \rightarrow & \GL_2^+(\RR)\\
  (u,x,y,\theta) & \longmapsto &
  \begin{pmatrix}u&0\\0&u\end{pmatrix}
  \begin{pmatrix}1&x\\0&1\end{pmatrix}
  \begin{pmatrix}y^{\frac12}&0\\0&y^{-\frac12}\end{pmatrix}
  \begin{pmatrix}
    \cos(\theta)&\sin(\theta)
    \\-\sin(\theta)&\cos(\theta)
  \end{pmatrix}
\end{array}
\]
alors $L$ est l'opérateur~:
\[
e^{-2i\theta} \left( -iy\frac{\partial}{\partial x}
  +y\frac{\partial}{\partial y}
  -\frac1{2i}\frac{\partial}{\partial\theta} \right)
\]

\subsubsection{Presque-holomorphie sur $\GL_2(\RR)$}

Les opérateurs de base étant définis, on les utilise maintenant pour
définir la notion de presque-holomorphie, suivant
Shimura~\cite{Shimura2000a}. On sait que le noyau de l'opérateur $L$
défini précédemment est constitué des fonctions holomorphes. On juge
une condition d'holomorphie trop restrictive pour les applications ;
on cherche donc une condition un peu moins contraignante.

On dit qu'une fonction $f$ de classe $\mathcal{C}^\infty$ sur
$\GL_2(\RR)$ est presque-holomorphe en degré $e\in\NN$ lorsque
$L^{e+1}f=0$. En particulier, on retrouve les fonctions holomorphes
comme étant presque-holomorphes en degré zéro.

Il reste à étendre cette définition à des fonctions de la forme qui
nous intéresse. Pour chaque $\ii\in I$, on notera $L_\ii$ l'opérateur
$L$ agissant sur la variable $\ii$. On dit que
$f:\GL_2(\AA)\rightarrow\CC$ est presque holomorphe en degré $e\in\NN$
lorsque l'on a $L_{\ii_1}\dots L_{\ii_{e+1}}f=0$ pour tout
$(\ii_1,\dots,\ii_{e+1})\in I^{e+1}$.

On sait d'après le lemme~13.3 de \cite{Shimura2000a} que les fonctions
presque-holomorphes sont polynomiales en les $y_\ii^{-1}$, à
coefficients holomorphes, et que les espaces de fonctions automorphes
presque-holomorphes sont encore de dimension finie (c'est le
lemme~14.3 de \cite{Shimura2000a}). Par ailleurs, si $F\neq\QQ$, alors
le principe de Koecher permet de ne pas devoir donner de condition de
presque-holomorphie aux pointes, car une telle condition est
automatiquement réalisée ; c'est le résultat discuté dans le
paragraphe~14.15 de \cite{Shimura2000a} (où le nom n'est pas rappelé).

\subsubsection{Presque-holomorphie sur $\GL_2(\AA)$}

Les fonctions que l'on considère dans cet article sont définies sur
$\GL_2(\AA)$ ; la théorie précédente ne s'applique donc pas
directement, mais nécessite de petites adaptations.

Étant donnée $f:\GL_2(\AA)\rightarrow\CC$, on peut pour chaque place
$\ii\in I$ la considérer comme une fonction définie sur $\GL_2(\RR)$,
en fixant toutes les autres composantes. Cela permet d'une part de
définir la notion de classe $\mathcal{C}^\infty$ pour $f$, et d'autre
part de faire agir $L$ (que l'on notera alors $L_\ii$ pour bien
marquer qu'il s'agit de $L$ le long de la composante $\ii$).

On dira que $f$ est presque-holomorphe en degré $e\in\NN$ lorsque pour
tout choix de $(e+1)$-uplet de $I$, disons $(\ii_1,\dots,\ii_{e+1})$,
on a $L^{\ii_1}\dots L^{\ii_{e+1}}f=0$. En particulier, en degré zéro,
$f$ est holomorphe le long de chaque place infinie.

La presque-holomorphie se traduit au niveau des coefficients par le
fait suivant : si $\ii\in I$ est une place infinie le long de laquelle
on veut prouver la propriété, alors sur $\GL_2^+(\RR)$, si $g$
correspond à un point $\tau$ du demi-plan de Poincaré, alors $a_1(g)$
est, du point de vue de $\tau$, de la forme
$y^{-\nu}p(y)\exp(2i\pi\tau)$, avec $\nu\in\NN$ et $p$ un polynôme ;
le cas holomorphe étant bien sûr celui où $\nu=0$ et $p$ une
constante. Cette caractérisation est la plus aisée pour le cas qui
nous intéresse.

\subsubsection{Structure rationnelle}\label{sec:phol_rationn}

On dit qu'une forme modulaire est à coefficients dans un sous-anneau
$R$ de $\CC$ lorsque les expressions discutées précédemment font
intervenir des polynômes à coefficients dans $R$ ; on obtient ainsi
une structure rationnelle sur les formes modulaires
presque-holomorphes~: si $R_1\subset R_2$, alors les formes à
coefficients dans $R_1$ forment un sous-$R_1$-module des formes à
coefficients dans $R_2$.

Cette notion généralise la notion usuelle de structure rationnelle sur
les formes modulaires, où on dit qu'une forme est à coefficients dans
$R$ lorsqu'elle provient d'un élément de $R[[q]]$ via l'évaluation
$q=\exp(2i\pi\tau)$.

On trouvera dans~\cite{Hida1993a}, une discussion similaire de
\emph{deux} structures rationnelles sur un espace de formes
modulaires. On a choisi d'utiliser ici la seconde, qu'il note d'un
$\mathfrak{m}$ page~142, plutôt que la première, qu'il note d'un
$\mathfrak{M}$ page~141, car c'est celle qui permet de faire porter la
condition uniquement sur la fonction de Whittaker.

\section{Définition analytique}\label{sec:def}

\subsection{Expression intégrale}\label{sec:def_int}

On fixe d'ores et déjà un choix de caractère de Hecke unitaire
$\omega$, appelé à être le caractère central des fonctions automorphes
considérées et $s\in\CC$ un paramètre complexe, qui servira à obtenir
les convergences et les prolongements analytiques.

On se donne $\varphi$ une fonction de Schwartz-Bruhat, à partir de
laquelle on définit tout d'abord une nouvelle fonction, dite fonction
thêta~:
\[
\Theta(\varphi):g\longmapsto\sum_{v\in F^2\backslash\{0\}}(g\varphi)(v)
\]

\begin{remark}
  Dans les livres de Bump~\cite{Bump1997a} et de
  Garrett~\cite{Garrett1990a}, une fonction très similaire est définie
  via une sommation sans la condition de vecteur non nul ; les calculs
  qui suivent font donc intervenir $\Theta-1$. On préfère ici plutôt
  l'idée de Godement~\cite{Godement1966a}, où la sommation est faite
  sur un groupe, qui évite donc de faire apparaître un terme
  correctif.
\end{remark}

On définit alors la série d'Eisenstein (version analytique) par~:
\[
E^{\rm ana}_{\omega,s}(\varphi):
g\longmapsto
|\det g|^{-s}\int_{\AA^\times/F^\times}\omega(t)|t|^{2s}
\Theta(\varphi)(g/t){\rm d}^\times t
\]

Les résultats suivants, obtenus par simple jeu d'écriture permettent
souvent de simplifier les calculs~:
\begin{lemma}
  \begin{eqnarray*}
    \Theta(\varphi)(g)
    &=&
    \Theta(g\varphi)(1_{\GL_2(\AA)})
    \\
    E^{\rm ana}_{\omega,s}(\varphi)(g)
    &=&
    |\det g|^{-s}E^{\rm ana}_{\omega,s}(g\varphi)(1_{\GL_2(\AA)})
  \end{eqnarray*}
\end{lemma}

\subsection{Convergence}\label{sec:conv}

Montrer la convergence de $E^{\rm ana}_{\omega,s}(\varphi)(g)$ pour
toute fonction-test $\varphi$ et toute matrice $g$ équivaut à montrer
la convergence de $E^{\rm ana}_{\omega,s}(\varphi)(1_{\GL_2(\AA)})$
pour toute fonction-test, ce qui constitue une première réduction.

Par ailleurs, on sait que l'on a
$\AA^\times/F^\times\simeq\AA_1^\times/F^\times\times]0;+\infty[$, où
l'on sait que le premier facteur est compact, ce qui constitue une
seconde réduction.

On peut écrire~:
\begin{eqnarray*}
  E^{\rm ana}_{\omega,s}(\varphi)(1_{\GL_2(\AA)})
  &=&
  \int_{\AA^\times/F^\times}\sum_{v\in F^2\backslash\{0\}}\omega(t)
  |t|^{2s}\varphi(tv){\rm d}^\times t
  \\
  &=&
  \int_{\AA_1^\times/F^\times}\int_0^{+\infty}
  \sum_{v\in F^2\backslash\{0\}}\omega(kt)
  |kt|^{2s}\varphi(ktv){\rm d}^\times k{\rm d}^\times t
  \\
  &=&
  \int_{\AA_1^\times/F^\times}\int_0^{+\infty}\sum_{v\in F^2\backslash\{0\}}
  \omega(kt)|t|^{2s}\varphi(ktv){\rm d}^\times k{\rm d}^\times t
  \\
\end{eqnarray*}

Comme $\omega$ est unitaire, il suffit par domination de discuter de
la double intégrabilité et de la sommabilité de
$(k,t,v)\mapsto|t|^{2\Re(s)}|\varphi(ktv)|$. L'intégration sur un
espace compact ne pose aucun problème ; elle consiste simplement à
tordre la fonction-test sans en changer les propriétés. La partie
non-archimédienne d'une fonction-test est à support compact, donc
limite les dénominateurs dans la somme. La partie archimédienne
décroît très rapidement, donc la somme converge ; pour la même raison
l'intégrale impropre converge en $+\infty$. Finalement, il ne reste
qu'à vérifier la convergence de cette intégrale impropre sur la borne
$0$ ; mais c'est une intégrale de Riemann, donc sa convergence est
acquise dès lors que $s$ est de partie réelle strictement positive.

Enfin, cette étude montre que la convergence est uniforme en $s$ dans
toutes les bandes verticales de la forme $m\leq\Re(s)\leq M$ où
$0<m\leq M$ ; on obtient donc une fonction holomorphe sur le domaine
$\Re(s)>0$.

Il reste à discuter de la convergence de la série
$\Theta(\varphi)(g)$. Elle équivaut comme précédemment à la
convergence de la série $\Theta(\varphi)(1_{\GL_2(\AA)})$. Mais cette
dernière est évidente, puisque l'on somme une fonction, $\varphi$, à
décroissance rapide, sur un réseau.

\begin{remark}
  Il faut noter que l'on utilise la mesure multiplicative ; c'est la
  raison pour laquelle la condition de convergence pour les intégrales
  de Riemann est pour $2\Re(s)>0$ et pas $2\Re(s)>-1$ comme on aurait
  pu s'y attendre.
\end{remark}

\subsection{Expression sommatoire}\label{sec:def_somme}

Une telle expression est cruciale pour utiliser la technique du
déroulement dans les intégrales de Rankin et faire le lien entre
produit scalaire de Petersson et fonctions $L$ via les séries
d'Eisenstein.

On définit~:
\[
\zeta_{\omega,s}(\varphi)(g)
=
|\det g|^{-s} \int_{\AA^\times}|t|^{2s} \omega(t)
(g\varphi)\left(\binom{t}{0}\right) {\rm d}^\times t
\]

et on peut alors écrire~:
\begin{lemma}\label{lem:expr_sommatoire}
  \[
  E^{\rm ana}_{\omega,s}(\varphi)(g)
  =
  \sum_{\gamma\in B(F)\backslash \GL_2(F)}\zeta_{\omega,s}(\varphi)(\gamma g)
  \]
\end{lemma}

\begin{proof}
  Remarquons tout d'abord que la convergence n'est pas évidente. Le
  calcul se déroule de façon plus agréable en allant de la somme sur
  le quotient de groupes vers les séries d'Eisenstein, et c'est la
  raison pour laquelle on a choisi de présenter le calcul dans ce
  sens. Mais c'est la convergence de l'expression finale (qui plus est
  absolue) que l'on a justifié précédemment. La convergence de
  l'expression initiale se justifie donc en constatant que les
  différentes étapes respectent la convergence au fur et à mesure.

  On commence par substituer la définition intégrale de
  $\zeta_{\omega,s}$ dans la somme, puis on interprète le facteur
  $|t|^{2s}$ comme le déterminant de l'homothétie de rapport $t$, ce
  qui permet en faisant passer le quotient à un scalaire près de la
  somme à l'intégrale, puis en permutant, de se ramener à l'intégrale
  d'une somme, ce qui amène déjà très près du résultat final~:
  \begin{eqnarray*}
    \lefteqn{
      \sum_{\gamma\in B(F)\backslash \GL_2(F)}\zeta_{\omega,s}(\varphi)(\gamma g)
    }\\
    &=&
    \sum_{\gamma\in B(F)\backslash \GL_2(F)}
    |\det (\gamma g)|^{-s}\int_{\AA^\times}|t|^{2s}\omega(t)
    \varphi\left(g^{-1}\gamma^{-1}\binom{t}{0}\right)\rm d^\times t
    \\
    &=&
    \sum_{\gamma\in B(F)\backslash \GL_2(F)}\int_{\AA^\times}
    \left|\det\left(\frac1t\gamma g\right)\right|^{-s}
    \omega(t)\varphi\left(g^{-1}\left(\frac1t\gamma\right)^{-1}
      \binom{1}{0}\right)\rm d^\times t
    \\
    &=&
    \sum_{\gamma\in\left\{\begin{pmatrix}1&\star\\0&\star\end{pmatrix}\right\}\backslash \GL_2(F)}\int_{\AA^\times/F^\times}
    \left|\det\left(\frac1t\gamma g\right)\right|^{-s}
    \omega(t)\varphi\left(g^{-1}\left(\frac1t\gamma\right)^{-1}
      \binom{1}{0}\right)\rm d^\times t
    \\
    &=&
    |\det g|^{-s}
    \int_{\AA^\times/F^\times}
    \omega(t)|t|^{2s}
    \sum_{\gamma\in\left\{\begin{pmatrix}1&\star\\0&\star\end{pmatrix}\right\}\backslash \GL_2(F)}
    \varphi\left(g^{-1}\left(\frac1t\gamma\right)^{-1}
      \binom{1}{0}\right)\rm d^\times t
    \\
    &=&
    |\det g|^{-s}
    \int_{\AA^\times/F^\times}
    \omega(t)|t|^{2s}
    \sum_{\gamma\in\left\{\begin{pmatrix}1&\star\\0&\star\end{pmatrix}\right\}\backslash \GL_2(F)}
    \varphi\left(g^{-1}\gamma^{-1}\binom{t}{0}\right)\rm d^\times t
    \\
  \end{eqnarray*}

  Il reste pour conclure à faire le lien entre la sommation qui
  apparaît dans cette dernière expression et celle qui définit
  $\Theta$ ; ce point est éclairci par la bijection explicite
  suivante~:
  \begin{eqnarray*}
    F^2\backslash\{0\}
    &\longrightarrow&
    \left\{\begin{pmatrix}1&\star\\0&\star\end{pmatrix}\right\}
    \backslash \GL_2(F)
    \\
    g^{-1}\binom10
    &\longleftarrow\mapstochar&
    g
    \\
    (v_1,v_2)
    &\longmapsto&
    \left\{\begin{array}{cc}
        \begin{pmatrix}1/v_1 & 0 \\ -v_2 & v_1\end{pmatrix} & (v_1\neq0) \\
        \begin{pmatrix}0 & 1/v_2 \\ -v_2 & v_1\end{pmatrix} & (v_1=0) \\
      \end{array}\right.
  \end{eqnarray*}

  On peut maintenant terminer le calcul précédemment entamé~:
  \begin{eqnarray*}
    \lefteqn{
      \sum_{\gamma\in B(F)\backslash \GL_2(F)}\zeta_{\omega,s}(\varphi)(\gamma g)
    }\\
    &=&
    |\det g|^{-s}
    \int_{\AA^\times/F^\times}
    \omega(t)|t|^{2s}
    \sum_{\gamma\in\left\{\begin{pmatrix}1&\star\\0&\star\end{pmatrix}\right\}\backslash \GL_2(F)}
    \varphi\left(g^{-1}\gamma^{-1}\binom{t}{0}\right)\rm d^\times t
    \\
    &=&
    |\det g|^{-s}
    \int_{\AA^\times/F^\times}
    \omega(t)|t|^{2s}
    \sum_{(v_1,v_2)\in F^2\backslash\{0\}}
    \varphi\left(g^{-1}\binom{tv_1}{tv_2}\right)\rm d^\times t
    \\
    &=&
    |\det g|^{-s}
    \int_{\AA^\times/F^\times}
    \omega(t)|t|^{2s}
    \Theta(\varphi)(g/t)
    \rm d^\times t
    \\
    &=&
    E_{\omega,s}^{\rm ana}(\varphi)(g)
    \\
  \end{eqnarray*}
\end{proof}

\section{Prolongement analytique et équation
  fonctionnelle}\label{sec:ana_eq_fun}

\subsection{Lemme de symétrie}

À nouveau, la relation $E^{\rm ana}_{\omega,s}(\varphi)(g)=|\det
g|^{-s}E^{\rm ana}_{\omega,s}(g\varphi)(1_{\GL_2(\AA)})$ permet de
ramener la discussion à l'étude de $s\mapsto E^{\rm
  ana}_{\omega,s}(\varphi)(1_{\GL_2(\AA)})$.

Ceci étant dit, on écrit~:
\begin{eqnarray*}
  E^{\rm ana}_{\omega,s}(\varphi)(1_{\GL_2(\AA)})
  &=&
  \int_{\AA^\times/F^\times}\omega(t)|t|^{2s}\Theta(\varphi)(1/t){\rm d}^\times t
  \\
  &=&
  \int_{\AA_{<1}^\times/F^\times}\omega(t)|t|^{2s}
  \Theta(\varphi)(1/t){\rm d}^\times t
  \\
  & &
  +\int_{\AA_{>1}^\times/F^\times}\omega(t)|t|^{2s}
  \Theta(\varphi)(1/t){\rm d}^\times t
  \\
  &=&
  \int_{\AA_{<1}^\times/F^\times}\omega(t)|t|^{2s}
  \Theta(\varphi)(1/t){\rm d}^\times t
  \\
  & &
  +\int_{\AA_{>1}^\times/F^\times}\omega(t)|t|^{2s}
  \Theta(\varphi)(1/t){\rm d}^\times t
\end{eqnarray*}

On se concentre sur la seconde intégrale, dans laquelle on applique
d'abord la formule de Poisson~:
\begin{eqnarray*}
  \lefteqn{
    \int_{\AA_{<1}^\times/F^\times}\omega(t)|t|^{2s}
    \Theta(\varphi)(1/t){\rm d}^\times t
  }\\
  &=&
  \int_{\AA_{<1}^\times/F^\times}\omega(t)|t|^{2s}
  \sum_{v\in F^2\backslash\{0\}}\varphi(tv){\rm d}^\times t
  \\
  &=&
  \int_{\AA_{<1}^\times/F^\times}\omega(t)|t|^{2s}
  \left(\sum_{v\in F^2}\varphi(tv)
    -\varphi(0)\right)
  {\rm d}^\times t
  \\
  &=&
  \int_{\AA_{<1}^\times/F^\times}\omega(t)|t|^{2s}
  \sum_{v\in F^2}\varphi(tv){\rm d}^\times t
  \\
  & &
  -\int_{\AA_{<1}^\times/F^\times}\omega(t)|t|^{2s}\varphi(0){\rm d}^\times t
  \\
  &=&
  \int_{\AA_{<1}^\times/F^\times}\omega(t)|t|^{2s}
  \sum_{v\in F^2}|t|^{-2}\hat\varphi(v/t)
  {\rm d}^\times t
  \\
  & &
  -\int_{\AA_{<1}^\times/F^\times}\omega(t)|t|^{2s}\varphi(0){\rm d}^\times t
\end{eqnarray*}
on inverse alors la variable~:
\begin{eqnarray*}
  \int_{\AA_{<1}^\times/F^\times}\omega(t)|t|^{2s}
  \Theta(\varphi)(1/t){\rm d}^\times t
  &=&
  \int_{\AA_{<1}^\times/F^\times}\omega(t)|t|^{2s}
  \sum_{v\in F^2\backslash\{0\}}|t|^{-2}\hat\varphi(v/t)
  {\rm d}^\times t
  \\
  & &
  +\int_{\AA_{<1}^\times/F^\times}\omega(t)|t|^{2s}|t|^{-2}
  \hat\varphi(0){\rm d}^\times t
  \\
  & &
  -\int_{\AA_{<1}^\times/F^\times}\omega(t)|t|^{2s}\varphi(0){\rm d}^\times t
  \\
  &=&
  \int_{\AA_{>1}^\times/F^\times}
  \bar\omega(t)|t|^{2-2s}\sum_{v\in F^2\backslash\{0\}}\hat\varphi(tv)
  {\rm d}^\times t
  \\
  & &
  +\hat\varphi(0)\int_{\AA_{<1}^\times/F^\times}\omega(t)|t|^{2(s-1)}{\rm d}^\times t
  \\
  & &
  -\varphi(0)\int_{\AA_{<1}^\times/F^\times}\omega(t)|t|^{2s}{\rm d}^\times t
  \\
  &=&
  \int_{\AA_{>1}^\times/F^\times}
  \bar\omega(t)|t|^{2-2s}\Theta(\hat\varphi)(1/t)
  {\rm d}^\times t
  \\
  & &
  +\hat\varphi(0)\int_{\AA_{<1}^\times/F^\times}\omega(t)|t|^{2(s-1)}{\rm d}^\times t
  \\
  & &
  -\varphi(0)\int_{\AA_{<1}^\times/F^\times}\omega(t)|t|^{2s}{\rm d}^\times t
\end{eqnarray*}

Il ne reste qu'à appliquer le lemme~\ref{int_ideles_str_1} pour
contrôler complètement cette deuxième intégrale~:
\begin{eqnarray*}
  \int_{\AA_{<1}^\times/F^\times}\omega(t)|t|^{2s}
  \Theta(\varphi)(1/t){\rm d}^\times t
  &=&
  \int_{\AA_{>1}^\times/F^\times}
  \bar\omega(t)|t|^{2-2s}\Theta(\hat\varphi)(1/t) {\rm d}^\times t
  \\
  & &
  +\frac{\delta(\omega)\hat\varphi(0)}{2(s-1)}
  -\frac{\delta(\omega)\varphi(0)}{2s}
\end{eqnarray*}

On a finalement prouvé~:
\begin{lemma}\label{eq_func}
  \begin{eqnarray*}
    E^{\rm ana}_{\omega,s}(\varphi)(1_{\GL_2(\AA)})
    &=&
    \int_{\AA_{>1}^\times/F^\times}\omega(t)|t|^{2s}
    \Theta(\varphi)(1/t){\rm d}^\times t
    \\
    & &
    +\int_{\AA_{>1}^\times/F^\times}
    \bar\omega(t)|t|^{2-2s}\Theta(\hat\varphi)(1/t)
    {\rm d}^\times t
    \\
    & &
    +\frac{\delta(\omega)}2
    \left(
      \frac{\hat\varphi(0)}{s-1}
      -\frac{\varphi(0)}{s}
    \right)
  \end{eqnarray*}
\end{lemma}

\subsection{Prolongement analytique}\label{sec:analytique}

Comme $\delta(\omega)=\delta(\bar\omega)$, la relation précédente
exprime que la série d'Eisenstein admet un prolongement méromorphe au
plan complexe, avec d'éventuels pôles simples en $s=0$ et $s=1$, avec
l'équation fonctionnelle restreinte~:
\[
E^{\rm ana}_{\omega,s}(\varphi)(1_{\GL_2(\AA)})
=
E^{\rm ana}_{\bar\omega,1-s}(\hat\varphi)(1_{\GL_2(\AA)})
\]

\subsection{Équation fonctionnelle}\label{sec:eqfun}

On souhaite obtenir une équation fonctionnelle plus générale. Pour
cela, on fixe une matrice $g\in\GL_2(\AA)$, et on montre, pour
$(x_1,x_2)\in\AA^2$, via un simple changement de variable dans
l'intégrale~:
\[
\widehat{g\varphi}\binom{x_1}{x_2}
=
\left|\det g\right|
\left(g^\sharp\hat\varphi\right)
\left(\binom{x_1}{x_2}\right)
\]
où on a défini~: $g^\sharp=\frac{1}{\det g}g$.

À partir de cette expression, on déduit aisément~:
\[
\forall g\in\GL_2(\AA),
E^{\rm ana}_{\omega,s}(\varphi)(g)
=
E^{\rm ana}_{\bar\omega,1-s}(\hat\varphi)(g^\sharp)
\]

\section{Propriétés d'automorphie}\label{sec:autom}

\subsection{$\GL_2(F)$-invariance à gauche}\label{sec:autom_F_inv}

Si on a $\gamma\in\GL_2(F)$, alors comme $\gamma^{-1}$ induit une
bijection de $F^2\backslash\{0\}$ (sur lui même), on peut calculer~:
\begin{eqnarray*}
  \Theta(\varphi)(\gamma g)
  &=&
  \Theta(g\varphi)(\gamma)
  \\
  &=&
  \sum_{v\in F^2\backslash\{0\}}(g\varphi)(\gamma^{-1}v)
  \\
  &=&
  \sum_{v\in F^2\backslash\{0\}}(g\varphi)(v)
  \\
  &=&
  \Theta(\varphi)(g)
\end{eqnarray*}

Comme de plus $|\det\gamma g|=|\det g|$, il vient~:
\[
\forall\gamma\in\GL_2(F),
E_{\omega,s}^{\rm ana}(\varphi)(\gamma g) = E_{\omega,s}^{\rm ana}(\varphi)(g)
\]

\subsection{Action du centre}\label{sec:autom_centre}

On calcule, pour $z\in\AA^\times$, via un simple changement de
variable $t=zu$ dans l'intégrale~:
\begin{eqnarray*}
  E^{\rm ana}_{\omega,s}(\varphi)(zg)
  &=&
  |\det (zg)|^{-s}\int_{\AA^\times/F^\times}\omega(t)|t|^{2s}
  \Theta(\varphi)(zg/t)\rm d^\times t
  \\
  &=&
  |\det g|^{-s}|z|^{-2s}
  \int_{\AA^\times/F^\times}\omega(zu)|zu|^{2s}\Theta(\varphi)(g/u)\rm d^\times u
  \\
  &=&
  \omega(z)E^{\rm ana}_{\omega,s}(\varphi)(g)
\end{eqnarray*}

D'où la seconde propriété d'automorphie~:
\[
\forall z\in\AA^\times,
E_{\omega,s}^{\rm ana}(\varphi)(z g) = \omega(z)E_{\omega,s}^{\rm ana}(\varphi)(g)
\]

\subsection{Choix des poids}\label{sec:autom_poids}

En théorie des représentations, il est d'usage de demander que les
fonctions soient $K$-finies à droite, pour un choix de sous-groupe
compact $K$ défini par produit sur les places. Pour les places
infinies réelles, le choix usuel est de considérer $SO_2(\RR)$. Or,
une fonction $SO_2(\RR)$-finie à droite s'écrit comme somme finie de
fonctions propres pour les caractères de ce groupe ; comme il
s'identifie à $\CC/\ZZ$, son dual de Pontryagin s'identifie à $\ZZ$ :
tout caractère s'écrit donc $\theta\mapsto e^{ik\theta}$ où
$k\in\ZZ$. Ce paramètre $k$ est un poids, que l'on voudra positif, et
on va donc s'arranger pour qu'en chaque place infinie on puisse le
choisir commodément.

Pour cela, on prescrit la partie archimédienne de la fonction-test de
la façon suivante ; pour chaque $\ii\in I$, on fixe $k_\ii\in\mathbb
N$ et on définit~:
\[
\varphi_\ii:
(v_1,v_2)\longmapsto
\exp\left(-\pi(v_1^2+v_2^2)\right)(v_1+iv_2)^{k_\ii}
\]
et on suppose $\varphi$ de la forme $\varphi_f\times\prod_{\ii\in I}\varphi_\ii$.

Si on fixe maintenant des angles $(\theta_\ii)_{\ii\in I}$, et une
matrice adélique dont la partie non-archimédienne est l'identité~:
\[
\rho=\left(
  \begin{pmatrix}
    \cos(\theta_\ii) & \sin(\theta_\ii) \\
    -\sin(\theta_\ii) & \cos(\theta_\ii)
  \end{pmatrix}
\right)_{\ii\in I}
\]

On a tout d'abord, par un calcul assez élémentaire~:
\begin{eqnarray*}
  \lefteqn{
    \rho_\ii\varphi_\ii\left(\binom{v_1}{v_2}\right)
  }
  \\
  &=&
  \varphi_\ii\left(
    \begin{pmatrix}
      \cos(\theta_\ii) & -\sin(\theta_\ii) \\
      \sin(\theta_\ii) & \cos(\theta_\ii)
    \end{pmatrix}
    \binom{v_1}{v_2}
  \right)
  \\
  &=&
  \exp\left(-\pi(v_1^2+v_2^2\right)
  \left(
    \cos(\theta_\ii)v_1i\sin(\theta_\ii)v_2
    +i\sin(\theta_\ii)v_1+i\cos(\theta_\ii)v_2
  \right)^{k_\ii}
  \\
  &=&
  e^{ik_\ii\theta_\ii}\varphi_\ii\left(\binom{v_1}{v_2}\right)
\end{eqnarray*}

et donc, plus globalement, on a alors~:
\begin{eqnarray*}
  E^{\rm ana}_{\omega,s}(\varphi)(g\rho)
  &=&
  |\det g|^{-s}E^{\rm ana}_{\omega,s}(g\rho\varphi)(1_{\GL_2(\AA)})
  \\
  &=&
  \left(\prod_{\ii\in I}e^{ik_\ii\theta_\ii}\right)
  E^{\rm ana}_{\omega,s}(g\varphi)(1_{\GL_2(\AA)})
  \\
  &=&
  \left(\prod_{\ii\in I}e^{ik_\ii\theta_\ii}\right)
  E^{\rm ana}_{\omega,s}(\varphi)(g)
\end{eqnarray*}

Ce qui fournit la troisième propriété d'automorphie~:
\[
E_{\omega,s}^{\rm ana}(\varphi)(g\rho)
=
\left(\prod_{\ii\in I}e^{ik_\ii\theta_\ii}\right) E_{\omega,s}^{\rm ana}(\varphi)(g)
\]

\subsection{Niveau}\label{sec:autom_niveau}

Étant donnée une fonction-test $\varphi$, on sait qu'il existe un
sous-groupe d'indice fini $K\subset \GL_2(\hat{\mathcal O})$ tel que~:
\[
\forall h\in K, \varphi(gh)=\varphi(g)
\]
(où l'on identifie $h$ qui n'est qu'adélique finie avec une matrice
adélique en décidant que les composantes infinies sont l'identité)

On peut alors calculer, pour $h\in K$~:
\begin{eqnarray*}
  E^{\rm ana}_{\omega,s}(\varphi)(gh)
  &=&
  |\det gh|^{-s}
  E^{\rm ana}_{\omega,s}(gh\varphi)(1_{\GL_2(\AA)})
  \\
  &=&
  |\det g|^{-s}
  E^{\rm ana}_{\omega,s}(g\varphi)(1_{\GL_2(\AA)})
  \\
  &=&
  E^{\rm ana}_{\omega,s}(\varphi)(g)
\end{eqnarray*}

Ce qui fournit la quatrième propriété d'automorphie~:
\[
E_{\omega,s}^{\rm ana}(\varphi)(gh) = E_{\omega,s}^{\rm ana}(\varphi)(g)
\]

\subsection{Croissance modérée}\label{sec:autom_cr_moderee}

La notion de croissance modérée pour une fonction
$f:\GL_2(\AA)\rightarrow\CC$ est la suivante~: pour tout compact
$C\subset \GL_2(\AA)$, il existe $M>0$ et $\sigma>0$ tels que~:
\[
\forall(g_1,g_2)\in C^2,
\forall y\in\AA^\times,
\left|f\left(g_1\begin{pmatrix}y&0\\0&1\end{pmatrix}g_2\right)\right|
\leq
M\left(|y|^\sigma+|y|^{-\sigma}\right)
\]

On souhaite prouver une telle domination pour les séries d'Eisenstein
étudiées ici, avec de plus un contrôle localement uniforme en le
paramètre $s$ (mieux: localement uniforme en sa partie réelle).

On sait que l'on a pour tout $g\in\GL_2(\AA)$ (c'est la relation du
lemme~\ref{eq_func}, page~\pageref{eq_func})~:
\begin{eqnarray*}
  |\det g|^sE^{\rm ana}_{\omega,s}(\varphi)(g)
  &=&
  \int_{\AA_{>1}^\times/F^\times}\omega(t)|t|^{2s}
  \Theta(g\varphi)(1/t){\rm d}^\times t
  \\
  & &
  +\int_{\AA_{>1}^\times/F^\times}
  \bar\omega(t)|t|^{2-2s}\Theta(\widehat{g\varphi})(1/t)
  {\rm d}^\times t
  \\
  & &
  +\frac{\delta(\omega)}2
  \left(
    \frac{\widehat{g\varphi}(0)}{1-s}
    -\frac{(g\varphi)(0)}{s}
  \right)
\end{eqnarray*}

Dans cette expression, les deux termes intégraux sont de même nature,
et les termes restants vérifient clairement la condition de croissance
uniforme voulue ; on va donc se concentrer sur la croissance modérée
(et localement uniforme en $\Re(s)$) de~:
\[
g
\longmapsto
\int_{\AA_{>1}^\times/F^\times}\omega(t)|t|^{2s}\Theta(g\varphi)(1/t){\rm d}^\times t
\]

On se donne donc un compact $C\subset \GL_2(\AA)$, $(g_1,g_2)\in C^2$
et $y\in\AA^\times$ ; et on domine brutalement~:
\begin{eqnarray*}
  \lefteqn{
    \left|
      \int_{\AA_{>1}^\times/F^\times}
      \omega(t)|t|^{2s}
      \Theta\left(
        g_1\begin{pmatrix}y&0\\0&1\end{pmatrix}g_2\varphi
      \right)(1/t)
      {\rm d}^\times t
    \right|
  }
  \\
  &\leq&
  \int_{\AA_{>1}^\times/F^\times}
  |t|^{2\Re(s)}
  \left|
    \Theta\left(
      g_1\begin{pmatrix}y&0\\0&1\end{pmatrix}g_2\varphi
    \right)(1/t)
  \right|
  {\rm d}^\times t
  \\
  &\leq&
  \int_{\AA_{>1}^\times/F^\times}
  |t|^{2\Re(s)}
  \left|
    \sum_{v\in F^2\backslash\{0\}}
    \left(
      g_1\begin{pmatrix}y&0\\0&1\end{pmatrix}g_2\varphi
    \right)(tv)
  \right|
  {\rm d}^\times t
  \\
  &\leq&
  \int_{\AA_{>1}^\times/F^\times}
  |t|^{2\Re(s)}
  \sum_{v\in F^2\backslash\{0\}}
  \left|
    \left(
      g_1\begin{pmatrix}y&0\\0&1\end{pmatrix}g_2\varphi
    \right)(tv)
  \right|
  {\rm d}^\times t
  \\
  &\leq&
  \int_{\AA_{>1}^\times/F^\times}
  |t|^{2\Re(s)}
  \sum_{v\in F^2\backslash\{0\}}
  \left|
    \varphi
    \left(
      g_2^{-1}\begin{pmatrix}y^{-1}&0\\0&1\end{pmatrix}g_1^{-1}tv
    \right)
  \right|
  {\rm d}^\times t
\end{eqnarray*}

Dans cette expression, on sait que l'on a un isomorphisme~:
\[
\AA_{>1}^\times/F^\times\simeq]1;+\infty[\times\AA_1^\times/F^\times
\]

où le second facteur est compact ; la seule source de croissance dans
l'intégrale est donc le premier facteur.

La partie non-archimédienne de $\varphi$ est à support compact ; la
perturbation par le couple $(g_1,g_2)$ balayant un compact, on peut
dominer par l'indicatrice d'un compact invariant par $C$ à gauche et à
droite. Cette indicatrice restreint la somme sur $v$ à une somme sur
un réseau.

Par ailleurs, comme $\varphi$ a une partie archimédienne à
décroissance rapide, on sait que l'on peut dominer par l'inverse d'une
puissance arbitrairement grande de la norme, suffisamment grande pour
compenser l'intégrale et la sommation sur le réseau, et ce de façon
localement uniforme en $\Re(s)$.

On sait que pour les séries d'Eisenstein, une telle notion de
croissance est ce que l'on peut en général espérer de mieux ;
normalement, on demande aux formes automorphes de vérifier une
condition de type \og carré intégrable\fg{}, mais les séries
d'Eisenstein ne rentrent pas directement dans ce cadre (voir à ce
sujet la discussion en page~347 de Bump~\cite{Bump1997a}).

\subsection{Conditions de compatibilité}

On se place dans l'hypothèse où $\varphi$ est choisie pour obtenir une
notion de poids, et on souhaite mettre en évidence qu'un tel choix
doit être cohérent avec le choix du caractère central $\omega$. Si on
définit $z\in\AA$ par $1$ sur toutes les places, sauf $-1$ sur une
place archimedienne $\ii\in I$, alors en comparant l'action de la
matrice $\begin{pmatrix}z&\\&z\end{pmatrix}$ vue comme élément du
centre et comme matrice de rotation, il vient~:
\[
\omega_\ii(-1)=(-1)^{k_\ii}
\]

\section{Coefficients automorphes}\label{sec:coeffs}

\subsection{Forme générale du développement}

Si on fixe $g\in\GL_2(\AA)$, la fonction~:
\[
x
\longmapsto
E_{\omega,s}^{\rm ana}(\varphi)\left(\begin{pmatrix}1&x\\0&1\end{pmatrix}g\right)
\]
est définie sur $\AA$ et $F$-périodique ; par dualité de Pontryagin,
on peut donc écrire~:
\[
E_{\omega,s}^{\rm ana}(\varphi)\left(\begin{pmatrix}1&x\\0&1\end{pmatrix}g\right)
=
\sum_{\xi\in F}a_\xi(g)\Psi(\xi x)
\]
où l'on a~:
\[
a_\xi(g)
=
\int_{\AA/F}
E_{\omega,s}^{\rm ana}(\varphi)\left(\begin{pmatrix}1&u\\0&1\end{pmatrix}g\right)
\Psi(-\xi u){\rm d}u
\]

En pratique, on sait que tous les coefficients pour $\xi\neq0$ peuvent
s'exprimer en termes de $a_1$, qui est appelée la fonction de
Whittaker de la forme automorphe ; en effet, pour $\xi\neq0$, on peut
écrire, en utilisant successivement un changement de variable puis la
$\GL_2(F)$-invariance~:
\begin{eqnarray*}
  a_1\left(\begin{pmatrix}\xi&0\\0&1\end{pmatrix}g\right)
  &=&
  \int_{\AA/F}
  E_{\omega,s}^{\rm ana}(\varphi)\left(
    \begin{pmatrix}1&u\\0&1\end{pmatrix}
    \begin{pmatrix}\xi&0\\0&1\end{pmatrix}
    g
  \right)\Psi(-u){\rm d}u
  \\
  &=&
  \int_{\AA/F}
  E_{\omega,s}^{\rm ana}(\varphi)\left(
    \begin{pmatrix}\xi&u\\0&1\end{pmatrix}
    g
  \right)\Psi(-u){\rm d}u
  \\
  &=&
  \int_{\AA/F}
  E_{\omega,s}^{\rm ana}(\varphi)\left(
    \begin{pmatrix}\xi&\xi v\\0&1\end{pmatrix}
    g
  \right)\Psi(-\xi v)|\xi|{\rm d}v
  \\
  &=&
  \int_{\AA/F}
  E_{\omega,s}^{\rm ana}(\varphi)\left(
    \begin{pmatrix}\xi&0\\0&1\end{pmatrix}
    \begin{pmatrix}1&v\\0&1\end{pmatrix}
    g
  \right)\Psi(-\xi v){\rm d}v
  \\
  &=&
  \int_{\AA/F}
  E_{\omega,s}^{\rm ana}(\varphi)\left(
    \begin{pmatrix}1&v\\0&1\end{pmatrix}
  \right)\Psi(-\xi v){\rm d}v
  \\
  &=&
  a_\xi(g)
\end{eqnarray*}

\subsection{Calcul de la fonction de Whittaker}\label{sec:coeffs_Whittaker}

On écrit~:
\begin{eqnarray*}
  a_1(g)
  &=&
  \int_{\AA/F}
  E_{\omega,s}^{\rm ana}(\varphi)
  \left(\begin{pmatrix}1&u\\0&1\end{pmatrix}g\right)\Psi(-u){\rm d}u
  \\
  &=&
  \int_{\AA/F}|\det g|^{-s}\int_{\AA^\times/F^\times}\omega(t)|t|^{2s}
  \Theta(g\varphi)
  \left(\frac1t\begin{pmatrix}1&u\\0&1\end{pmatrix}\right)
  {\rm d^\times}t\Psi(-u){\rm d}u
  \\
  &=&
  \int_{\AA^\times/F^\times}|\det g|^{-s}\omega(t)|t|^{2s}
  \int_{\AA/F}\Theta(g\varphi)
  \left(\frac1t\begin{pmatrix}1&u\\0&1\end{pmatrix}\right)
  \Psi(-u)
  {\rm d}u
  {\rm d^\times}t
\end{eqnarray*}

Ce qui amène naturellement à calculer cette intégrale sur $\Theta$~:
\begin{eqnarray*}
  \lefteqn{\int_{\AA/F}
    \Theta(g\varphi)\left(\frac1t\begin{pmatrix}1&u\\0&1\end{pmatrix}\right)
    \Psi(-u)
    {\rm d}u
  }\\
  &=&
  \int_{\AA/F}
  \sum_{(v_1,v_2)\in F^2\backslash\{0\}}
  (g\varphi)\left(t\begin{pmatrix}1&-u\\0&1\end{pmatrix}\binom{v_1}{v_2}\right)
  \Psi(-u)
  {\rm d}u
\end{eqnarray*}

On scinde la sommation en d'une part une sommation sur $v_1\in
F^\times,v_2=0$, de couple $(v_1,0)$ et d'autre part une sommation sur
$(v_1,v_2)\in F\times F^\times$, de couple $(v_1v_2,v_2)$~:
\begin{eqnarray*}
  \lefteqn{
    \int_{\AA/F}
    \Theta(g\varphi)\left(\frac1t\begin{pmatrix}1&u\\0&1\end{pmatrix}\right)
    \Psi(-u)
    {\rm d}u
  }
  \\
  &=&
  \sum_{v_1\in F^\times}
  \int_{\AA/F}
  (g\varphi)\left(t\begin{pmatrix}1&-u\\0&1\end{pmatrix}\binom{v_1}{0}\right)
  \Psi(-u)
  {\rm d}u
  \\
  & &
  +
  \sum_{(v_1,v_2)\in F\times F^\times}
  \int_{\AA/F}
  (g\varphi)\left(t\begin{pmatrix}1&-u\\0&1\end{pmatrix}
    \binom{v_1v_2}{v_2}\right)
  \Psi(-u)
  {\rm d}u
  \\
  &=&
  \sum_{v_1\in F^\times}
  (g\varphi)\left(t\binom{v_1}{0}\right)
  \left(\int_{\AA/F}\Psi(-u){\rm d}u\right)
  \\
  & &
  +
  \sum_{(v_1,v_2)\in F\times F^\times}
  \int_{\AA/F}
  (g\varphi)\left(tv_2\binom{v_1-u}1\right)
  \Psi(-u)
  {\rm d}u
\end{eqnarray*}

Dans cette dernière expression, le premier terme est nul, et le second
se déroule en une intégrale sur $\AA$, car on sait
$\Psi(-u)=\Psi(v_1-u)$~:
\begin{eqnarray*}
  \int_{\AA/F}
  \Theta(\varphi)\left(\frac1t\begin{pmatrix}1&u\\0&1\end{pmatrix}g\right)
  \Psi(-u)
  {\rm d}u
  &=&
  \sum_{v_2\in F^\times}
  \int_{\AA}
  (g\varphi)\left(tv_2\binom{-u}1\right)
  \Psi(-u)
  {\rm d}u
\end{eqnarray*}

Ce calcul permet alors d'écrire~:
\begin{eqnarray*}
  a_1(g)
  &=&
  |\det g|^{-s}
  \int_{\AA^\times/F^\times}
  \omega(t)|t|^{2s}
  \sum_{v_2\in F^\times}
  \int_{\AA}
  (g\varphi)\left(tv_2\binom{-u}1\right)
  \Psi(-u)
  {\rm d}u
  {\rm d^\times}t
  \\
  &=&
  |\det g|^{-s}
  \int_{\AA^\times/F^\times}
  \omega(t)|t|^{2s}
  \sum_{v_2\in F^\times}
  \int_{\AA}
  (g\varphi)\left(tv_2\binom{u}1\right)
  \Psi(u)
  {\rm d}u
  {\rm d^\times}t
\end{eqnarray*}

où l'on constate à nouveau que la somme et l'intégrale se combinent
pour donner une seule intégrale, car on sait
$\omega(t)|t|^{2s}=\omega(tv_2)|tv_2|^{2s}$ ; d'où la formule finale~:
\begin{equation}\label{frml:Whit_int}
  a_1(g)
  =
  |\det g|^{-s}
  \int_{\AA\times\AA^\times}\omega(t)|t|^{2s}
  (g\varphi)\left(t\binom{u}1\right) \Psi(u) {\rm d}u{\rm d^\times}t
\end{equation}

Cette expression intégrale montre que si $\varphi$ est une
fonction-test factorisable sur toutes les places (de telles fonctions
sont parfois appelées \og purs tenseurs\fg{}), alors la fonction de
Whittaker se factorise en produit d'intégrales locales. C'est du reste
ainsi que l'on prouvera la presque-holomorphie~: en se concentrant sur
les intégrales locales en les places infinies.

\begin{remark}
  Cette formule est comparable à la formule donnée en page~434 par
  Scholl dans~\cite{Scholl1998a}, pour $B_\chi^\star(g)$ : l'intégrale
  qu'il donne n'est que semi-adélique, et il a donc des facteurs à
  l'infini (une fonction $\Gamma$ sous la forme d'une factorielle et
  des puissances de $2i\pi$), qui sont ici encore cachés par la partie
  archimédienne des intégrales adéliques.
\end{remark}

\section{Presque-holomorphie}\label{sec:phol}

On va prouver que sous certaines hypothèses, les séries d'Eisenstein
sont presque-holomorphes.  Pour arriver à un tel résultat, on va
procéder par réductions successives à des cas plus simples à traiter.

\subsection{Réduction sur $\varphi$}

Pour un choix de $\varphi$ ``élémentaire'', c'est-à-dire se
factorisant selon toutes les places (certains auteurs parlent de \og
pur tenseur\fg{}), et telle que la série d'Eisenstein est de poids
$(k_\ii)_{\ii\in I}$, on va calculer explicitement la fonction de
Whittaker, ce qui montrera la presque-holomorphie.

Les deux restrictions du paragraphe précédent ne sont pas trop
limitatives, dans la mesure où toute fonction-test est une combinaison
linéaire de telles fonctions, et où il faut bien s'être donné des
poids pour contrôler la partie archimédienne.

\subsection{Réduction à une place $\ii\in I$}

On part de la formule intégrale~\ref{frml:Whit_int},
page~\pageref{frml:Whit_int}, qui est entièrement factorisable ; on
est donc ramené à étudier, pour $\ii\in I$~:
\[
a_{1,\ii}:
g
\longmapsto
|\det(g)|^{-s}
\int_{\RR\times\RR^\times}
\omega_\ii(t)|t|^{2s}
(g\varphi_\ii)
\left(
  t\binom{u}{1}
\right)
\exp(2i\pi u)
{\mathrm d}u{\mathrm d}^\times t
\]
où comme on l'a dit~:
\[
\varphi_{\infty,\ii}:(x_1,x_2)\mapsto(x_1+ix_2)^{k_\ii}\exp(-(x_1^2+x_2^2))
\]

\subsection{Réduction à $\GL_2^+(\RR)$}

La $\GL_2(F)$ invariance a comme conséquence que~:
\[
E^{\rm ana}_{\omega,s}\left(
  \begin{pmatrix}
    -1 & 0 \\
    0  & 1
  \end{pmatrix}
  g\right)
=
E^{\rm ana}_{\omega,s}\left(g\right)
\]

On peut donc supposer que l'on s'est arrangé pour que le calcul le
long de la place $\ii$ selon laquelle on souhaite prouver la
presque-holomorphie se fasse dans la composante connexe de l'identité.

\subsection{Réduction au demi-plan de Poincaré}

Étant donnée $g\in\GL_2^+(\RR)$, on sait que l'on peut la décomposer
sous la forme~:
\[
g=
\begin{pmatrix}\delta&0\\0&\delta\end{pmatrix}
\begin{pmatrix}1&x\\0&1\end{pmatrix}
\begin{pmatrix}y^{1/2}&0\\0&y^{-1/2}\end{pmatrix}
\begin{pmatrix}
  \cos(\theta)&\sin(\theta)\\
  -\sin(\theta)&\cos(\theta)
\end{pmatrix}
\]
où $\delta>0$, $x\in\RR$, $y>0$ et $\theta\in\RR/2\pi\ZZ$.

Un simple changement de variable en ``$t$'' dans l'intégrale suffit à
se ramener au cas où $\delta=1$.

Il est aussi assez aisé de constater que~:
\[
a_{1,\ii}(g)
=
\exp(ik_\ii\theta)
a_{1,\ii}\left(
  \begin{pmatrix}1&x\\0&1\end{pmatrix}
  \begin{pmatrix}y^{1/2}&0\\0&y^{-1/2}\end{pmatrix}
\right)
\]

On est donc ramené au cas d'une matrice de $SL_2(\RR)$ représentant un
élément du demi-plan de Poincaré ; notons-le (traditionnellement)
$\tau=x+iy$.

\subsection{Traitement du cas du demi-plan de Poincaré}

On commence par déterminer l'action de $\tau$ sur $\varphi_\ii$ ; pour
$t\in\RR^\times$ et $u\in\RR$~:
\begin{eqnarray*}
  \lefteqn{
    \left(
      \begin{pmatrix}1&x\\0&1\end{pmatrix}
      \begin{pmatrix}y^{1/2}&0\\0&y^{-1/2}\end{pmatrix}
      \varphi_{\infty,\ii}
    \right)
    \left(
      t
      \binom{u}{1}
    \right)
  }\\
  &=&
  (ty^{-1/2})^k((u-x)+iy)^k\exp(-(ty^{-1/2})^2((u-x)^2+y^2))
\end{eqnarray*}

D'où, d'abord par une translation suivant $u$, puis en utilisant les
conditions de compatibilité qui donnent $\omega_\ii(t)t^k=|t|^k$ et
permettent le passage d'une intégrale sur $\RR^\times$ à une intégrale
sur $\RR_+^\times$~:
\begin{eqnarray*}
  \lefteqn{
    a_{1,\ii}\left(
      \begin{pmatrix}1&x\\0&1\end{pmatrix}
      \begin{pmatrix}y^{1/2}&0\\0&y^{-1/2}\end{pmatrix}
    \right)
  }\\
  &=&
  \int_{\RR\times\RR^\times}
  \omega_\ii(t)|t|^{2s}
  (ty^{-1/2})^k((u-x)+iy)^k
  \\
  & &
  \times
  \exp(-(ty^{-1/2})^2((u-x)^2+y^2))
  \exp(2i\pi u)
  {\mathrm d}u{\mathrm d}^\times t
  \\
  &=&
  2\exp(2i\pi x)y^{-k/2}
  \\
  & &
  \times
  \int_{\RR\times\RR_+^\times}
  t^{k+2s}
  \exp(-t^2(u^2+y^2)/y)
  \\
  & &
  \times
  (u+iy)^k
  \exp(2i\pi u)
  {\mathrm d}u{\mathrm d}^\times t
  \\
\end{eqnarray*}

L'intégrale le long de $t$ se calcule via le lemme~\ref{lem:int_Gamma}
en page~\pageref{lem:int_Gamma} ; on arrange ensuite l'expression pour
préparer l'étape suivante~:
\begin{eqnarray*}
  \lefteqn{
    a_{1,\ii}\left(
      \begin{pmatrix}1&x\\0&1\end{pmatrix}
      \begin{pmatrix}y^{1/2}&0\\0&y^{-1/2}\end{pmatrix}
    \right)
  }\\
  &=&
  2\exp(2i\pi x)y^{-k/2}
  \\
  & &
  \times
  \int_{\RR}
  \frac12
  \frac{\Gamma\left(s+\frac{k}2\right)}{\left((u^2+y^2)/y\right)^{s+\frac{k}2}}
  (u+iy)^k
  \exp(2i\pi u)
  {\mathrm d}u
  \\
  &=&
  \exp(2i\pi x)y^s\Gamma\left(s+\frac{k}2\right)
  \\
  & &
  \times
  \int_{\RR}
  \frac{1}{(u^2+y^2)^{s+\frac{k}2}}
  (u+iy)^k
  \exp(2i\pi u)
  {\mathrm d}u
  \\
  &=&
  \exp(2i\pi x)y^s\Gamma\left(s+\frac{k}2\right)
  \\
  & &
  \times(-1)^k
  \int_{\RR}
  \frac{1}{(u^2+y^2)^{s+\frac{k}2}}
  (u-iy)^k
  \exp(-2i\pi u)
  {\mathrm d}u
  \\
  &=&
  \exp(2i\pi x)y^s\Gamma\left(s+\frac{k}2\right)(-1)^k
  \\
  & &
  \times
  \int_{\RR}
  (u+iy)^{-\left(s+\frac{k}2\right)}
  (u-iy)^{-\left(s-\frac{k}2\right)}
  \exp(-2i\pi u)
  {\mathrm d}u
\end{eqnarray*}

On est alors en position d'appliquer le lemme~\ref{lem:int_conf_hyper}
(page~\pageref{lem:int_conf_hyper})~:
\begin{eqnarray*}
  \lefteqn{
    a_{1,\ii}\left(
      \begin{pmatrix}1&x\\0&1\end{pmatrix}
      \begin{pmatrix}y^{1/2}&0\\0&y^{-1/2}\end{pmatrix}
    \right)
  }\\
  &=&
  \exp(2i\pi x)y^s\Gamma\left(s+\frac{k}2\right)(-1)^k
  \\
  & &
  \times
  \exp(-2\pi y)i^k\Gamma\left(s+\frac{k}2\right)^{-1}(2\pi)^{2s}
  (4\pi y)^{-\left(s-\frac{k}2\right)}
  \omega\left(4\pi y,s+\frac{k}2,s-\frac{k}2\right)
  \\
  &=&
  \exp(2i\pi(x+iy))(-i)^ky^s(2\pi)^{2s}
  (4\pi y)^{-\left(s-\frac{k}2\right)}
  \omega\left(4\pi y,s+\frac{k}2,s-\frac{k}2\right)
  \\
\end{eqnarray*}

\subsection{Cas de presque-holomorphie}\label{sec:phol_cas}

Dans cette dernière égalité, si on fait le choix $s=\frac{k}2-r$ avec
$r\in\NN$, on a finalement~:
\begin{eqnarray*}
  \lefteqn{
    a_{1,\ii}\left(
      \begin{pmatrix}1&x\\0&1\end{pmatrix}
      \begin{pmatrix}y^{1/2}&0\\0&y^{-1/2}\end{pmatrix}
    \right)
  }\\
  &=&
  \exp(2i\pi(x+iy))(-i)^ky^{\frac{k}2-r}(2\pi)^{k-2r}
  (4\pi y)^r\omega(4\pi y,k-r,-r)
  \\
\end{eqnarray*}

Cette dernière expression fait finalement apparaître la fonction
confluente hypergéométrique $\omega$, avec des arguments pour lesquels
on sait qu'elle a un développement polynomial, que l'on sait
expliciter ; c'est la proposition~\ref{prop:conf_hyper_poly} de la
page~\pageref{prop:conf_hyper_poly}.

\section{Exemple explicite}\label{sec:explicite}

Le but de cette section est d'écrire plus explicitement les calculs
qui précèdent pour les éclairer. On va donc se placer sur le corps
$\QQ$, fixer un poids $k\geq2$ et $s\in\CC$ de partie réelle
strictement positive, pour que la convergence soit assurée sans
prolongement analytique.

On définit à partir de là une fonction $\varphi$ de Schwartz-Bruhat en
demandant en partie non-archimédienne l'indicatrice de $\hat\ZZ^2$, et
en place infinie (unique vu le choix du corps) la fonction
correspondant au choix du poids $k$ ; enfin, on décide de considérer
un caractère central trivial ($\omega=1$).

On fixe enfin $x\in\RR$ et $y>0$, à partir desquels on définit
$g\in\GL_2(\AA)$ par l'identité en partie finie et
$\begin{pmatrix}1&x\\0&1\end{pmatrix}
\begin{pmatrix}y^{1/2}&0\\0&y^{-1/2}\end{pmatrix} $ en partie infinie,
et $z=x+iy$ dans le demi-plan de Poincaré.

On peut alors calculer :
\begin{eqnarray*}
\lefteqn{ E^{\rm ana}_{1,s}(\varphi)(g)} \\
  &=&
  |\det(g)|^{-s}
  \int_{\AA^\times/\QQ^\times}
  |t|^{2s}\Theta(\varphi)\left(\frac gt\right)
  {\rm d}^\times t
  \\
  &=&
  \int_{\AA^\times/\QQ^\times}|t|^{2s}
  \sum_{v\in\QQ^2\backslash\{0\}}\varphi(tg^{-1}v)
  {\rm d}^\times t
  \\
  &=&
  \int_{\AA^\times/\QQ^\times}
  |t|^{2s}
  \sum_{v\in\QQ^2\backslash\{0\}}
  \varphi_f(t_fv)
  \varphi_\infty
  \left(
    t_\infty
    \begin{pmatrix}y^{-1/2}&0\\0&y^{1/2}\end{pmatrix}
    \begin{pmatrix}1&-x\\0&1\end{pmatrix}
    \binom{v_1}{v_2}
  \right)
  {\rm d}^\times t
  \\
  &=&
  \int_{\AA^\times/\QQ^\times}
  |t|^{2s}
  \sum_{v\in\QQ^2\backslash\{0\}}
  \varphi_f(t_fv)
  \varphi_\infty
  \left(
    t_\infty y^{-1/2}
    \begin{pmatrix}1&0\\0&y\end{pmatrix}
    \begin{pmatrix}1&-x\\0&1\end{pmatrix}
    \binom{v_1}{v_2}
  \right)
  {\rm d}^\times t
  \\
  &=&
  \int_{\AA^\times/\QQ^\times}
  |t|^{2s}
  \sum_{v\in\QQ^2\backslash\{0\}}
  \varphi_f(t_fv)
  \varphi_\infty
  \left(
    t_\infty y^{-1/2}
    \binom{v_1-v_2x}{v_2y}
  \right)
  {\rm d}^\times t
  \\
  &=&
  \int_{\AA^\times/\QQ^\times}
  |t|^{2s}
  \sum_{v\in\QQ^2\backslash\{0\}}
  \varphi_f(t_fv)
  t_\infty^ky^{-k/2}
  (v_1-v_2\bar z)^k
  \exp\left(-\pi t_\infty^2y^{-1}|v_1-v_2\bar z|^2\right)
  {\rm d}^\times t
  \\
  &=&
  y^{-k/2}\sum_{v\in\QQ^2\backslash\{0\}}
  \int_{\AA^\times/\QQ^\times}
  |t|^{2s}t_\infty^k
  \ind_{\hat\ZZ^2}(t_fv)
  (v_1-v_2\bar z)^k
  \exp\left(-\pi t_\infty^2y^{-1}|v_1-v_2\bar z|^2\right)
  {\rm d}^\times t
  \\
\end{eqnarray*}

On utilise alors l'isomorphisme entre $\AA^\times/\QQ^\times$ et
$\AA_1^\times/\QQ^\times\times]0;+\infty[$ déjà utilisé pour écrire :
\begin{eqnarray*}
  E^{\rm ana}_{1,s}(\varphi)(g)
  &=&
  y^{-k/2}\sum_{v\in\QQ^2\backslash\{0\}}
  \left(\int_{\AA_1^\times/\QQ^\times}
    \ind_{\hat\ZZ^2}(tv)
    {\rm d}^\times t\right)
  (v_1-v_2\bar z)^k
  \\
  & &
  \times
  \left(\int_0^{+\infty}
    t^{k+2s}
    \exp\left(-\pi t^2y^{-1}|v_1-v_2\bar z|^2\right)
    {\rm d}^\times t\right)
  \\
\end{eqnarray*}

d'où l'on déduit, en utilisant l'intégrale archimédienne du
lemme~\ref{lem:int_Gamma} et l'intégrale non-archimédienne du
lemme~\ref{lem:vol_ideles_1} :
\begin{eqnarray*}
  E^{\rm ana}_{1,s}(\varphi)(g)
  &=&
  y^{-k/2}\sum_{v\in\QQ^2\backslash\{0\}}
  \ind_{\hat\ZZ^2}(v)
  {\rm Vol}(\AA_1^\times/\QQ^\times)
  (v_1-v_2\bar z)^k
  \\
  & &
  \times
  \frac12
  \frac{\Gamma(s+\frac{k}2)}{(\pi y^{-1}|v_1-v_2\bar z|^2)^{s+\frac{k}2}}
  \\
  &=&
  y^{-s}
  \frac{\Gamma\left(s+\frac{k}2\right)}{\pi^{s+\frac{k}2}}
  \sum_{v\in\ZZ^2\backslash\{0\}}
  (v_1-v_2\bar z)^k
  |v_1-v_2\bar z|^{-k-2s}
\end{eqnarray*}

Notons que pour le choix $s=\frac{k}2$, qui correspond à l'exemple le
plus classique, on obtient~:
\[
E^{\rm ana}_{1,\frac{k}2}(\varphi)(g) = y^{-\frac{k}2}
\frac{\Gamma(k)}{\pi^k} \sum_{v\in\ZZ^2\backslash\{0\}} (v_1+v_2
z)^{-k}
\]

\begin{remark}
  La technique usuelle de prolongement analytique donne les fonctions
  usuelles avec le choix $s=0$ en général ; cependant, on a ici ramené
  l'équation fonctionnelle à une symétrie par rapport à $1$ ; c'est
  donc bien avec le choix $s=\frac{k}2$ qu'il est naturel de retrouver
  les séries les plus classiques.
\end{remark}

\bibliography{../Abords/bibliography}
\bibliographystyle{plain-fr}

\end{document}